\documentclass[a4paper,12pt]{article}
\usepackage{amscd,amsthm,amsfonts,amscd,epsf,amsmath,amssymb,latexsym,multicol}
\usepackage[latin1]{inputenc}
\usepackage{pstricks,pst-node,cite,epsfig}
\usepackage{graphicx}
\usepackage{xcolor}
\usepackage{mathrsfs}
\usepackage{textcomp}
\usepackage{setspace}
\usepackage{array}
\usepackage{float}
\usepackage{fancyhdr}
\usepackage{url}
\usepackage{pgf,tikz}
\usetikzlibrary{arrows,automata,positioning,decorations.pathreplacing,decorations.markings}

\usepackage{cite,tabularx,
caption,fancyhdr,geometry,color,indentfirst}

\begin{document}

\theoremstyle{definition}
\newtheorem{dfn}{Definition}[section]
\newtheorem{thm}{Theorem}
\newtheorem{lem}[thm]{Lemma}
\newtheorem {pro}[thm]{Proposition}
\newtheorem{cor}[thm]{Corollary}

\newtheorem{rmk}{Remark}

\title{\textbf{Normal category of partitions of a set}}
\author{A R Rajan \footnote{The author acknowledges the financial support of the Kerala State Council for Science, Technology and Environment, Trivandrum (via the award of Emeritus scientist) in the preparation of this article. } \\ \small Department Of Mathematics, University Of Kerala, \\ \small Trivandrum, Kerala-695581, India. \\ \small \emph{arrunivker@yahoo.com } \and 
P A Azeef Muhammed \footnote{The author wishes to acknowledge the financial support of the Council for Scientific and Industrial Research, New Delhi (via JRF and SRF) in the preparation of this article.}\\ \small Department Of Mathematics, University Of Kerala, \\ \small Trivandrum, Kerala-695581, India. \\ \small \emph{azeefp@gmail.com} }

\date{April 2015 }
\maketitle

\begin{abstract}
Let $T_X$ be the semigroup of all non-invertible transformations on an arbitrary set $X$. It is known that $T_X$ is a regular semigroup. The principal right(left) ideals of a regular semigroup $S$ with partial left(right) translations as morphisms form a normal category $\mathcal{R}( S )$($\mathcal{L}( S )$). Here we consider the category $\Pi(X)$ of partitions of a set $X$ and show that it admits a normal category structure and that $\Pi(X)$ is isomorphic to the category $\mathcal{R}( T_X )$. We also consider the normal dual $N^\ast \mathscr{P}(X)$ of the power-set category $\mathscr{P}(X)$ associated with $X$ and show that $N^\ast \mathscr{P}(X)$ is isomorphic to the partition category - $\Pi(X)$ of the set $X$.\\[.2cm]
\end{abstract}
The concept of normal categories has been introduced by K S S Nambooripad (cf. \cite{cross}) in the context of describing cross-connections for regular semigroups. Cross-connection is a method for constructing regular semigroups from the categories of principal left ideals and principal right ideals. 
\section{Preliminaries}
We assume familiarity with the definitions and elementary concepts of category theory (cf. \cite{mac}). In the following, the definitions and results on normal categories are as in \cite{cross}. For a category $\mathcal{C}$, we denote by $v\mathcal{C}$ the set of objects of $\mathcal{C}$. If $f:a \to b$ and $g:b \to c$ are morphisms in $\mathcal{C}$, then the composition of $f$ and $g$ gives $f\circ g : a \to c$ in $\mathcal{C}$.
\begin{dfn}
Let $\mathcal{C}$ and $\mathcal{D}$ be two categories and $F:\mathcal{C} \to \mathcal{D}$ be a functor. We denote by $vF$ the induced map from $v\mathcal{C}$ to $v\mathcal{D}$. We shall say that a functor $F$ is \emph{v-injective} if $vF$ is injective. $F$ is said to be \emph{v-surjective} if $vF$ is surjective. $F$ is said to be an isomorphism if it is $v$-injective, $v$-surjective, full and faithful. 
\end{dfn}
\begin{dfn}
A \emph{preorder} $\mathcal{P}$ is a category such that for any $ p , p' \in v\mathcal{P} $, the hom-set $\mathcal{P}(p,p')$ contains atmost one morphism.
\end{dfn}
In this case the relation $\subseteq$ on the class $v\mathcal{P}$ of objects of $\mathcal{P}$ defined by
$$ p\subseteq p'\iff \mathcal{P}(p,p')\ne\emptyset $$ is a quasi-order. $\mathcal{P}$ is said to be a strict preorder if $\subseteq$ is a partial order.
\begin{dfn}
Let $\mathcal{C}$ be a category and $\mathcal{P}$ be a subcategory of $\mathcal{C}$. Then $(\mathcal{C} ,\mathcal{P})$ is called a \emph{category with subobjects} if the following hold:
\begin{enumerate}
\item $\mathcal{P}$ is a strict preorder with $v\mathcal{P} = v\mathcal{C}$.
\item Every $f\in \mathcal{P}$ is a monomorphism in $\mathcal{C}$.
\item If $f,g\in \mathcal{P}$ and if $f=hg$ for some $h \in \mathcal{C}$, then
  $h\in \mathcal{P}$.
\end{enumerate}
In a category with subobjects, if $f: c \to d $ is a morphism in $ \mathcal{P}$, then $f$ is said to be an \emph{inclusion}. And we denote this inclusion by $j(c,d)$.
\end{dfn}
In the following, $(\mathcal{C} ,\mathcal{P})$ is a category with subobjects.
\begin{dfn}
A morphism $e: d \to c$ in $\mathcal{C}$ is called a
\emph{retraction} if $c \subseteq d$ and $j(c,d) e = 1_{c}$.
\end{dfn}
\begin{dfn}
A \emph{normal factorization} of a morphism $f \in \mathcal{C}(c,d)$ is a
factorization of the form $f=euj$ where $e:c\to c'$ is a retraction,
$u:c'\to d'$ is an isomorphism and $j=j(d',d)$ for some $c',d' \in v\mathcal{C}$ with $ c' \subseteq c$, $ d' \subseteq d$. 
\end{dfn}
It may be noted here that normal factorization of a morphism is not unique. But if $f = euj = e'u'j'$ are two normal factorizations of f, then it can be shown that $ eu = e'u' $ and $j = j'$. And here we denote $eu$ by $f^\circ$. Observe that $f^\circ$ is independent of the factorization and that $f^\circ$ is an epimorphism. We call $f^\circ$ the epimorphic part of $f$.
\begin{dfn}\label{dfn1}
 Let $\mathcal{C}$ be a category and $d\in v\mathcal{C}$. A map $\gamma:v\mathcal{C}\to\mathcal{C}$ is called a \emph{cone from the base $v\mathcal{C}$ to the vertex $d$} (or simply a cone in $\mathcal{C}$ to $d$) if $\gamma$ satisfies
  the following:
  \begin{enumerate}
  \item $\gamma(c)\in \mathcal{C}(c,d)$ for all $c\in v\mathcal{C}$.
  \item If $c'\subseteq c$ then $j(c',c)\gamma(c) = \gamma(c')$.
  \end{enumerate}
\end{dfn}
  Given the cone $\gamma$ we denote by $c_{\gamma}$ the the \emph{vertex} of
  $\gamma$ and for each $c\in v\mathcal{C}$, the morphism $\gamma(c): c \to c_{\gamma}$
  is called the \emph{component} of $\gamma$ at $c$. We define $ M_\gamma = \{ c \in \mathcal{C}\: :\: \gamma(c)\text{ is an isomorphism} \}$.
\begin{dfn}\label{dfn2}
A cone $\gamma$ is said to
  be \emph{normal} if there exists $c\in v\mathcal{C}$ such that \\
  $\gamma(c):c\to c_{\gamma}$ is an isomorphism.
\end{dfn}
\begin{dfn}
A \emph{normal category} is a pair $(\mathcal{C}, \mathcal{P})$ satisfying the following :
\begin{enumerate}
\item $(\mathcal{C}, \mathcal{P})$ is a category with subobjects.
\item Any morphism in $\mathcal{C}$ has a normal factorization. 
\item For each $c \in v\mathcal{C} $ there is a normal cone $\sigma$ with vertex $c$ and $\sigma (c) = 1_c$.
\end{enumerate}
\end{dfn}
 Now we see that the normal cones in a normal category form a regular semigroup (cf. \cite{cross}). Let $\sigma$ be a normal cone with vertex $d$ and let $f : d\to d'$ be an epimorphism.
 Then $\sigma * f$ defined below is a normal cone.
 \begin{equation}
 (\sigma * f)(a) = \sigma(a)f
 \end{equation}
  \begin{thm}(cf. \cite{cross} )
 Let $(\mathcal{C}, \mathcal{P})$ be a normal category and let $T\mathcal{C}$ be the set of all normal cones in $\mathcal{C}$. Then $T\mathcal{C}$ is a regular semigroup with product defined as follows :\\
 For $\gamma, \sigma \in T\mathcal{C}$.
\begin{equation} \label{eqnsg}
(\gamma * \sigma)(a) = \gamma(a) (\sigma(c_\gamma))^\circ
\end{equation} 
 where $(\sigma(c_\gamma))^\circ$ is the epimorphic part of the $\sigma(c_\gamma)$.\\
 Then it can be seen that $\gamma * \sigma$ is a normal cone. $T\mathcal{C}$ is called the \emph{semigroup of normal cones} in $\mathcal{C}$.
 \end{thm}
 For each $\gamma \in T\mathcal{C}$, define $H({\gamma};-)$ on the objects and morphisms of $\mathcal{C}$ as follows. For each $c\in v\mathcal{C}$ and for each $g\in \mathcal{C}(c,d)$, define
  \begin{subequations} \label{eqnH}
    \begin{align}
      H({\gamma};{c})&= \{\gamma\ast f^\circ : f \in \mathcal{C}(c_{\gamma},c)\} \\
      H({\gamma};{g}):H({\gamma};{c})& \to H({\gamma};{d}) \text{ as } \gamma\ast f^\circ \mapsto \gamma\ast (fg)^\circ
    \end{align}
  \end{subequations}
\begin{pro}(cf. \cite{cross} ) \label{proH}
  For $\gamma, \gamma' \in T\mathcal{C}$ , $H(\gamma;-) = H(\gamma';-)$ if and only if there is a unique isomorphism $h: c_{\gamma'} \to c_\gamma$, such that $\gamma = \gamma' \ast h $. And $\gamma \mathscr{R} \gamma' \: \iff \: H(\gamma;-) = H(\gamma';-)$.
  \end{pro}
 \begin{dfn}
 If $\mathcal{C}$ is a normal category, then the \emph{normal dual} of $\mathcal{C}$ , denoted by $N^\ast \mathcal{C}$ , is the full subcategory of $\mathcal{C}^\ast $ with vertex set\\
 \begin{equation} \label{eqnH1}
 v N^\ast \mathcal{C} = \{ H(\epsilon;-) \: : \: \epsilon \in E(T\mathcal{C}) \}
 \end{equation} 
 where $\mathcal{C}^\ast$ is the category of all functors from $\mathcal{C}$ to $\bf{Set}$(cf. \cite{mac} ). 
 \end{dfn}
 \begin{thm} (cf. \cite{cross} )\label{thm1}
 To every morphism $\sigma : H(\epsilon;-) \to H(\epsilon ';-)$ in $N^\ast \mathcal{C} $, there is a unique $\hat{\sigma} : c_{\epsilon '} \to c_\epsilon$ in $\mathcal{C}$ such that the component of the natural transformation $\sigma$ at $c \in v\mathcal{C} $ is the map given by :
 \begin{equation} \label{eqnH2}
 \sigma(c) : \epsilon \ast f^\circ \mapsto \epsilon ' \ast (\hat{\sigma} f)^\circ 
 \end{equation} 
 Moreover $\sigma$ is the inclusion $ H(\epsilon;-) \subseteq H(\epsilon ';-)$ if and only if
 $ \epsilon = \epsilon ' \ast (\hat{\sigma})^\circ $.
 \end{thm}
Now we describe the normal category $\mathcal{R}(S)$ of the principal right ideals of a regular semigroup $S$. Since every principal right ideal in $S$ has at least one idempotent generator, we may write objects
(vertexes) in $\mathcal{R}(S)$ as $eS$ for $e\in E(S)$. A morphism $\lambda:eS\to fS$ is a left translation $\lambda=\lambda(e,s,f)$ where $s \in fSe$ and $\lambda$ maps $x \mapsto sx$. Thus
\begin{equation} \label{eqnRS}
  v\mathcal{R}(S) = \{ eS : e \in E(S)\}\quad\text{and}  \quad 
  \mathcal{R}(S) =\{\lambda(e,s,f) : e,f \in E(S),\; s\in fSe\}.    
\end{equation}
The following proposition gives the general properties of $\mathcal{R}(S)$.
\begin{pro} (cf. \cite{cross} )\label{pro1} Let $S$ be a regular semigroup. The $\mathcal{R}(S)$ is a normal category such that $\lambda(e,u,f)=\lambda(e',v,f')$ if and only if $e \mathscr{R} e'$, $f\mathscr{R} f'$, $u \in fSe$, $v\in f'Se'$ and $u=ve$. If $\lambda(e,u,f)$ and $\lambda(e',v,f')$ are composable morphisms in $\mathcal{R}(S)$ (so that $f \mathscr{R}e'$) and $u \in fSe$ and $v\in f'Se'$, then  $\lambda(e,u,f)\lambda(e',v,f') = \lambda(e,vu,f')$. Given a morphism $\lambda(e,u,f)$ in $\mathcal{R}(S)$, for any $g\in L_{u} \cap \omega(f)$ and $h\in E(R_{u})$,
 $$ \lambda=\lambda(e,g,g)\lambda(g,u,h)\lambda(h,h,f) $$
is a normal factorization of $\lambda$ and every normal factorization of $\lambda$ has this form.
\end{pro}
\begin{pro}(cf. \cite{cross} ) \label{pro2}
Let $S$ be a regular semigroup, $a \in S$ and $f \in E(L_a) $. Then for each $e \in E(S) $, let $\rho^a(Se) = \rho(e,ea,f) $. Then $\rho^a$ is a normal cone called the principal cone generaed by $a$ in $\mathcal{L}(S)$ with vertex $Sa$ such that $M_{\rho^a} = \{ Se : e \in E(R_a) \}$. $\rho^a$ is an idempotent in $T\mathcal{L}(S)$ iff $a \in E(S)$.
\end{pro}
\begin{pro}(cf. \cite{cross} )\label{pro3}
If S is a regular semigroup then the mapping $a \mapsto \rho^a$ is a homomorphism from S to $T\mathcal{L}(S)$. Further if S has an identity, then S is isomorphic to $T\mathcal{L}(S)$.
\end{pro}

\section{The category of partitions of a set}
A partition $\pi$ of $X$ is a family of subsets $A_i$ of $X$ such that $\bigcup A_i \: = \: X$ and $A_i\cap A_j = \phi$ for $i \ne j$. A partition is said to be non-identity if atleast one $A_i$ has more than one element. Any partition $\pi$ of $X$ determines an equivalence relation $\rho$ such that $\pi = X/\rho$, namely $a \rho b$ if and only if $a$ and $b$ belong to the same $A_i$ for some $i$. Conversely given any equivalence relation $\rho$, the family of sets $a\rho$ with $a$ in $X$ is a partition of $X$. For convenience, we denote this equivalence relation $\rho$ also by $\pi$ itself. Thus if $\pi$ is a partition of $X$, then we write $\pi$ for the equivalence relation on $X$ determined by the partition $\pi$. For a partition $\pi = \{A_i : i \in I\}$ and $a \in X$, we denote by $[a]_\pi$ the set $A_i$ such that $a \in A_i$\\
Given a non-identity partition $\pi$ of $X$, we denote by $\bar{\pi}$ the set of all functions from $\pi$  to $X$. If ${\eta}$ is a function from $\pi_2$ to $\pi_1$, we define $P_\eta : \bar{\pi_1}$ to $\bar{\pi_2}$ by $(\alpha)P_\eta = {\eta}\alpha$ for every $\alpha \in \bar{\pi_1}$.\\
Now we define the category of partitions $\Pi(X)$ of the set $X$ as follows. The vertex set is $v\Pi(X) = \{ \: \bar{\pi} \::\:\pi $ is a non-identity partition of $X \}$  and a morphism in $\Pi(X)$ from $\bar{\pi_1}$ to $\bar{\pi_2}$ is given by $P_\eta$ as defined above.\\ 
Define partial order on $\Pi(X)$ as follows. For $\pi_1 = \{ A_i : i \in I\}$ and $\pi_2 = \{ B_j : j \in J \}$ define $\bar{\pi_1} \leq \bar{\pi_2}$ if for each j, $B_j \subseteq A_i$ for some $i$. In this case, $\vartheta: B_j \mapsto A_i$ is a well-defined map from $\pi_2 \to \pi_1$ and $P_\vartheta : \bar{\pi_1} \to \bar{\pi_2}$ is a morphism in $\Pi(X)$. We consider $P_\vartheta$ as the inclusion morphism $P_\vartheta :\bar{\pi_1} \subseteq \bar{\pi_2}$. It can be observed that $\bar{\pi_1} \leq \bar{\pi_2}$ if and only if ${\pi_2} \subseteq {\pi_1}$ as equivalence relations. 
\begin{lem}\label{lrtr}
Let $\pi_1 = \{ A_i : i \in I\}$ and $\pi_2 = \{ B_j : j \in J \}$ be partitions of $X$ such that $\bar{\pi_1} \leq \bar{\pi_2}$. Let $P_\vartheta : \bar{\pi_1} \to \bar{\pi_2}$ be the inclusion. Then there exists a retraction $P_\zeta : \bar{\pi_2} \to \bar{\pi_1}$ .ie $j(\bar{\pi_1},\bar{\pi_2})P_\zeta = 1_{\bar{\pi_1}}$. 
\end{lem}
\begin{proof}
Since $\bar{\pi_1} \subseteq \bar{\pi_2}$, for each j, $B_j \subseteq A_i$ for some $i$. Define $\zeta: \pi_1 \to \pi_2$ as $(A_i)\zeta = B_j$ where $B_j$ is a subset chosen from $\{B_x \: : \: B_x \subseteq A_i \}$ . Clearly $P_\zeta$ is a morphism from $\bar{\pi_2}$ to $\bar{\pi_1}$. For $\alpha \in \bar{\pi_1}$, $(\alpha) j(\bar{\pi_1},\bar{\pi_2})P_\zeta = (\alpha) P_\vartheta P_\zeta = (\vartheta\alpha) P_\zeta = \zeta\vartheta\alpha $.\\
Now for any $A_i \in {\pi_1}$, $(A_i)\zeta\vartheta = A_i$. So $\zeta\vartheta : \pi_1 \to \pi_1$ is the identity map. Hence $(\alpha) j(\bar{\pi_1},\bar{\pi_2})P_\zeta = \alpha$ so that $j(\bar{\pi_1},\bar{\pi_2})P_\zeta : \bar{\pi_1} \to \bar{\pi_1}$ is the identity morphism. Thus $P_\vartheta P_\zeta = 1_{\bar{\pi_1}}$. And hence $P_\zeta$ is a retraction.
\end{proof}
\begin{pro} \label{propix}
$\Pi(X)$ is a normal category.
\end{pro}
\begin{proof}
By the above discussion, $\Pi(X)$ is a category with subobjects and every inclusion has an associated retraction as described above.\\
Now we prove the existence of normal factorization. Let $P_\eta$ from $\bar{\pi_1}$ to $\bar{\pi_2}$ be a morphism in $\Pi(X)$ so that $\eta$ is a mapping from $\pi_2$ to $\pi_1$. Let $\sigma_\eta$ be the equivalance relation on $X$ given by\\
$$ \sigma =\{ (x,y) \: :\: [x]\eta = [y]\eta \} $$
where $[x] \in \pi_2$. Then clearly $\pi_2 \subseteq \sigma$. Let $\vartheta : \pi_2 \to \sigma$ be the inclusion given by $[x] \mapsto [x\sigma]$ where $[x]$ is the equivalence class containing $x$ in $\pi_2$ and $[x\sigma]$ is the equivalence class containing $x$ in $\sigma$. Since $\vartheta: \pi_2 \to \sigma$ is the inclusion map, we see that $P_\vartheta : \bar{\sigma} \to \bar{\pi_2}$ is the inclusion morphism.\\
Let $\gamma = \gamma_\eta$ be the partition of $X$ defined as follows. Let $\pi_1 = \{ A_i : i \in I\}$ and Im $\eta =\{ A_i : i \in I'\}$ for some $I'\subseteq I$. Then fix an element $1 \in I'$ and define $\gamma := \{ B \cup A_1 , A_i : i \in I'\setminus\{1\} \}$ where $B = \cup\{ A_j : A_j \notin \text{Im }\eta \} $. Then clearly $\gamma$ is a partition of $X$ and $\pi_1 \subseteq \gamma$.\\
Let $\zeta : \gamma \to \pi_1$ be defined as follows.
\begin{equation*}
\begin{split}
A_1\cup& B\mapsto A_1\\
A_j& \mapsto A_j \text{ for } j \in I'\setminus\{1\}
\end{split}
\end{equation*}
Then we see that $P_\zeta$ is a retraction from $\bar{\pi_1}$ to $\bar{\gamma}$ as in the proof of lemma \ref{lrtr}. Now define $u: \sigma \to \gamma$ as follows. $u : [x\sigma] \mapsto [x]\eta$ where $[x]$ is the equivalence class of $\pi_2$ containing $x$. Clearly $u$ is well-defined and is a bijection. And since $u$ is a bijection, $P_u : \bar{\gamma} \to \bar{\sigma}$ is an isomorphism.\\
Now to see that $P_\eta = P_\zeta P_u P_\vartheta$ consider any $[x] \in \pi_2$; then $([x])\vartheta u \zeta\: = \: ([x\sigma])u \zeta \:=\: ([x]\eta)\zeta \:=\: [x]\eta$. And hence for any $\alpha \in \bar{\pi_1}$, $(\alpha)P_\eta = (\alpha)P_{\vartheta u \zeta} = \vartheta u \zeta\alpha = (\alpha)P_{\zeta}P_{u}P_{\vartheta}$. And consequently $P_\eta = P_{\zeta}P_{u}P_{\vartheta}$.\\
Given any partition $\pi$ of $X$, let $\sigma$ be a cone in $\Pi(X)$ with vertex $\bar{\pi}$ defined as follows. If $\pi = \{ A_{i} : i \in I\}$, let $u: \pi \to \pi $ be a mapping such that $u(A_i) \subseteq A_i$ for all $A_i \in \pi$. For any partition $\pi_f = \{ B_{j} : j \in J\}$, define $\nu : \pi \to {\pi_f} $ as $A_i \mapsto B_j$ such that $u(A_i) \subseteq B_j$. Then define $\sigma(\bar{\pi_f}) = \pi_\nu$. Clearly $\nu :\pi \to \pi$ is identity as $\pi_{\nu} = 1_{\bar{\pi}}$. Then $\sigma$ is a normal cone with vertex $\bar{\pi}$ and $\sigma(\bar{\pi}) = 1_{\bar{\pi}} $. Hence $\Pi(X)$ is a normal category.
\end{proof}
Let $X$ be a non-empty set. The transformation semigroup $T_X$ on $X$ is the semigroup of all non-invertible transformations on $X$. The partitions of $X$ can be related to the idempotents in $T_X$ as follows. Let $\pi$ be a partition of $X$ and let $A$ be a cross-section of $\pi$. Let $e: X \to X$ be defined by $e(x) = a$ where $a \in A$ and $x \in [a]_\pi$. Then $e$ is an idempotent in $T_X$ such that $\pi_e = \pi$ where $\pi_e$ is the partition determined by $e$.\\
Now we proceed to show that the category of right ideals of the transformation semigroup $T_X$ is isomorphic to $\Pi(X)$. The product of transformations is taken in the order it is written .i.e from left to right. For any $\alpha \in T_X$, we denote by $\pi_\alpha$ the partition of $X$ induced by $\alpha$. As equivalence relation, we may write $\pi_\alpha = \{ (x,y) : x\alpha = y\alpha \} $. Also $\pi_\alpha = \{ y \alpha^{-1} : y \in \text{Im }\alpha \} $.\vspace*{.2cm}\\ 
The following properties of $T_X$ will be used often.
\begin{lem} \label{tx} (cf. \cite{clif} )
Let $\alpha , \beta $ be arbitrary elements of $T_X$. Then the following statements hold. 
\begin{enumerate}
\item There exists $ \: \varepsilon \in T_X $ such that $\alpha\varepsilon = \beta $ if and only if $ X\alpha \supseteq X\beta $. Hence $\alpha \mathscr{L} \beta $ if and only if $ X\alpha = X\beta$.
\item There exists $ \: \varepsilon \in T_X $ such that $\varepsilon\alpha = \beta $ if and only if $ \pi_\alpha \subseteq \pi_\beta $. Hence $\alpha \mathscr{R} \beta $ if and only if $ \pi_\alpha = \pi_\beta$.
\end{enumerate}
\end{lem}
\noindent Now given $v \in fSe$, we define $$\eta_{v} : \pi_f \to \pi_e \text{ by }(xf^{-1})\eta_{v} \: = \: (xv)e^{-1} \text{ for } x \in \text{Im} f$$ 
\begin{lem} \label{lemv}
For $v \in fSe$, $\eta_{v}$ is a well defined function from $\pi_f$ to $\pi_{e}$.
\end{lem}
\begin{proof}
Let $x \in $ Im $f$, then $xf^{-1} \in \pi_f$. As $v \in fSe$, Im $v \subseteq$ Im $e$. And hence $xv \in \text{Im }e$ and $(xv)e^{-1} \in \pi_{e}$. And $\eta_{v}$ maps $\pi_f$ to $\pi_{e}$. Now if $xf^{-1} = yf^{-1}$ for $x,y \in$ Im $f$, then $x = y$. So $\eta_{v}$ is well-defined.
\end{proof}
\begin{rmk}\label{rmkpi}
Observe that the definition of $\eta_{v}$ depends on $v_{|\text{Im }f} : \text{Im }f \to \text{Im }e$ only. In general given any function $f$ from $A$ to $B$ such that $A$ and $B$ are cross-sections of the partitions $\pi_1$ and $\pi_2$; we can uniquely define a function $\eta_f$ from $\pi_1$ to $\pi_2$ as above. Conversely given a mapping ${\eta} : \pi_1 \to \pi_2 $, and for given cross-sections $A$, $B$ we can uniquely define an 'induced' mapping $f \in T_X$ such that $f_{|A}$ is a function from $A$ to $B$ and $\eta = \eta_{f}$.
\end{rmk}
\begin{lem}\label{lemv1}
Let $S = T_X$ and $e,f,g,h \in E(S)$. Let $v \in fSe$ and $u \in hSg$ be such that $\eta_{v} = \eta_{u}$. Then $\pi_f = \pi_h$, $\pi_e = \pi_g$ and $v = ue$. 
\end{lem}
\begin{proof}
By equating domains and codomains, we get $\pi_f = \pi_h$, $\pi_e = \pi_g$. Now since $\pi_f = \pi_h$, we have $f \mathscr{R} h$ in $S$ and so $fu = f(hu) = (fh)u = hu = u$. So $ue = fue \in fSe$. Now we see that $\eta_u = \eta_{ue}$. This follows from the observation that $g \mathscr{R} e$. That is $(xf^{-1})\eta_u = (xu)g^{-1} = (xug)g^{-1} = (xue)e^{-1} $(since $g \mathscr{R} e$) $= (xf^{-1})\eta_{ue} =(xf^{-1})\eta_{v}$. Now from the uniqueness in remark \ref{rmkpi}, we get $ue = v$.
\end{proof}
\begin{lem}\label{lemv2}
Let $v \in fSe$ and $u \in hSg$ be such that $\pi_f = \pi_h$, $\pi_e = \pi_g$ and $v = ue$. Then $\eta_{v} = \eta_{u}$. 
\end{lem}
\begin{proof}
Suppose $\pi_f = \pi_h$, $\pi_e = \pi_g$ and $v = ue$. Then $\eta_{v}$ and $\eta_{u}$ will be functions with domain $\pi_f = \pi_h$ and co-domain $\pi_e = \pi_g$. Since $v = ue$, as argued above $(xf^{-1})\eta_{v} = (xf^{-1})\eta_{ue} =  (xue)e^{-1}= (xug)g^{-1} = (xu)g^{-1} = (xf^{-1})\eta_u $ and hence $\eta_{v} = \eta_{u}$.
\end{proof}
\noindent By the above discussion, given a non-identity partition $\pi_e$ of $X$ induced by an idempotent transformation $e$ and $\bar{\pi_e}$ the set of all functions from $\pi_e$ to $X$, $P_v:\alpha \mapsto \eta_{v}\alpha$ for every $\alpha \in \bar{\pi_e}$ will be a morphism from $\bar{\pi_e}$ to $\bar{\pi_f}$ for $v \in fSe$. Hence the category $\Pi(X)$ of partitions can also be defined as follows. The vertex set $v\Pi(X) = \{ \: \bar{\pi_e} \::\:e \in E(T_X) \}$ and a morphism in $\Pi(X)$ from $\bar{\pi_e}$ to $\bar{\pi_f}$ is given by $P_v:\bar{\pi_e} \to \bar{\pi_f}$ for $v \in fSe$ defined by $\alpha \mapsto \eta_v\alpha$.
Now we show that $\Pi(X)$ is a category with subobjects with inclusions defined as follows.\\
When $\pi_f \subseteq \pi_e$ so that $e \in fSe$, then $P_e : \bar{\pi_e} \to \bar{\pi_f}$ are the inclusions. Here $\eta_e: \pi_f \to \pi_e$ maps $xf^{-1}$ to $(xe)e^{-1}$ and $\alpha \mapsto \eta_e\alpha$. 
Now we define a functor $G : \mathcal{R}(T_X) \to \Pi(X) $ as
\begin{equation} \label{eqnGP}
vG (eS) = \bar{\pi_e} \quad\text{and}  \quad
G(\lambda(e,v,f)) = P_v
\end{equation}
Observe $\lambda(e,v,f)$ is from $eS$ to $fS$ and $P_v$ is from $\bar{\pi_e}$ to $\bar{\pi_f}$. Now we need to show that $G$ is an inclusion preserving order-isomorphism. For that first we need to show that $G$ is a covariant functor.
\begin{lem} \label{lemG1}
$G$ as defined in equation \ref{eqnGP} is a well defined covariant functor from $\mathcal{R}(T_X)$ to $\Pi(X) $.
\end{lem}
\begin{proof}
Suppose $eS = fS$; then by lemma \ref{tx}, $\pi_e =\pi_{f}$. And hence $\bar{\pi_e} = \bar{\pi_f}$. Therefore $vG$ is well-defined. 
Now if $\lambda(e,v,f) = \lambda(g,u,h)$ then by Proposition \ref{pro1}, we get $e \mathscr{R} g$; and by lemma \ref{tx}, $\bar{\pi_e} = \bar{\pi_g}$. Similarly $\bar{\pi_f} = \bar{\pi_h}$. By Lemma \ref{lemv2} and Proposition \ref{pro1}, $\eta_v$ and $\eta_u$ are equal functions from $\pi_f$ to $\pi_{e}$. So $P_v = P_u$ and hence $G$ is well-defined on morphisms as well.\\
Now let $\lambda(e,v,f)$ and $\lambda(f,u,g)$ be two composable morphisms in $\mathcal{R}(T_X)$; then by Proposition \ref{pro1}, $G(\lambda(e,v,f)\lambda(f,u,g)) \:=\: G(\lambda(e,uv,g))$ $= \: P_{uv}$.\\
Also $G(\lambda(e,v,f))G(\lambda(f,u,g)) $ $= \: P_vP_{u}$ $= \: P_{uv}$.( Since $(\alpha)P_vP_{u} = (\eta_v\alpha)P_{u} = \eta_u\:\eta_v\alpha = \eta_{uv}\alpha = (\alpha)P_{uv} $ )\\
So $G(\lambda(e,v,f)\lambda(f,u,g)) \:=\: G(\lambda(e,v,f))G(\lambda(f,u,g)) $.\\ 
Hence $G$ is a covariant functor.
\end{proof}
\begin{lem} \label{lemG2}
$G$ is inclusion preserving.
\end{lem}
\begin{proof}
Suppose that $eS \subseteq fS$. Then by Lemma \ref{tx}, $\pi_{e} \supseteq \pi_{f} $ and hence $\bar{\pi_e} \leq \bar{\pi_f} $.\\
Also $G(j(eS,fS)) = G(\lambda(e,e,f)) = P_e$.\\
Now $P_e$ transforms each function in $\bar{\pi_e}$ as a function in $\bar{\pi_f}$. Hence $P_e$ is an inclusion in $\Pi(X)$ and $G$ is inclusion preserving.
\end{proof}
\begin{lem} \label{lemG2A}
$vG$ is an order isomorphism.
\end{lem}
\begin{proof}
Suppose that $eS \subseteq fS$. Then by Lemma \ref{tx}, $\pi_{e} \supseteq \pi_{f} $ and hence $\bar{\pi_e} \leq \bar{\pi_f} $.
Conversely if $\bar{\pi_e} \leq \bar{\pi_f} $, then $\pi_{e} \supseteq \pi_{f} $ and by lemma \ref{tx}, $eS \subseteq fS$.\\
Hence $eS \subseteq fS \iff \bar{\pi_e} \leq \bar{\pi_f} $ and so $vG$ is an order isomorphism.
\end{proof}
\begin{lem} \label{lemG3}
$G$ is $v-$surjective and full.
\end{lem}
\begin{proof}
Let $\bar{\pi} \in \Pi(X) $. Then it is clear to see that there exists an $e \in E(T_X)$ such that $\pi = \pi_e$ and thus $G(eS) = \bar{\pi_e}$. Hence $G$ is $v$-surjective.\\
Now let $P_\eta$ be a function from $\bar{\pi_e}$ to $\bar{\pi_f}$ where ${\eta}$ is a function from $\pi_f$ to $\pi_e$. Then $v_{|\text{Im}f}$ can be chosen to be a function from Im $f$ to Im $e$ by restricting ${\eta}$ to the cross-sections Im $f$ and Im $e$ of the partitions $\pi_f$ and $\pi_e$ respectively such that $\eta = \eta_v$ then $G(\lambda(e,v,f)) = P_v = P_\eta$.\\
Hence $G$ is full. 
\end{proof}
\begin{lem} \label{lemG4}
$G$ is $v$-injective and faithful.
\end{lem}
\begin{proof}
Let $G(eS)\:=\:G(fS) $ in $\Pi(X)$. That is $\bar{\pi_e} \: = \: \bar{\pi_f}$ and hence $\pi_e = \pi_f$. And by lemma \ref{tx}, $eS \:=\: fS$ and $G$ is $v$-injective.\\
Now let $G(\lambda(e,v,f)) = G(\lambda(g,u,h))$ in $\Pi(X)$. Then $P_v = P_{u} $ .i.e $\pi_{e} = \pi_{g}$, $\pi_f = \pi_h$ and $\eta_{v} = \eta_{u}$ .i.e $eS = gS$, $fS = hS$ and by lemma \ref{lemv1}, $v = ue$. Hence $\lambda(e,v,f) = \lambda(g,u,h)$ and $G$ is faithful. 
\end{proof}
\begin{thm} \label{thmpix}
$\mathcal{R}(T_X) $ is isomorphic to $\Pi(X)$ as normal categories.
\end{thm}
\begin{proof}
By the previous lemmas \ref{lemG1}, \ref{lemG2}, \ref{lemG2A}, \ref{lemG3}, \ref{lemG4} ; $G$ is an inclusion preserving covariant functor from $\mathcal{R}(T_X) $ to $\Pi(X)$ which is an order isomorphism, $v$-injective, $v$-surjective and fully-faithful.\\
Hence the theorem.
\end{proof}

\section{The normal dual of the power-set category}
Now we proceed to characterize the normal dual associated with the normal category $\mathcal{L}(  T_X )$. This dual is a normal category whose objects are $H$ functors and morphisms are natural transformations between the $H$ functors(see equations \ref{eqnH}, \ref{eqnH1}, \ref{eqnH2}). For that, we need the following characterization of $\mathcal{L}(  T_X )$.
\begin{thm}(cf. \cite{ltx} )\label{thm4}
The power-set category $\mathscr{P}(X)$ of a set $X$ is the category of proper subsets of a set $X$ with functions as morphisms. It is a normal category and $\mathcal{L}(T_X) $ is isomorphic to $\mathscr{P}(X)$ as normal categories.
\end{thm}
All the normal cones in $\mathscr{P}(X)$ are principal cones(cf. \cite{ltx}) and so the cones can also be represented by $\alpha \in T_X$ by identifying as follows. For $\alpha \in T_X$ and $\rho^\alpha \in T\mathscr{P}(X)$, $\rho^\alpha(A)\:=\: \alpha_{|A}$ for $A \in \mathscr{P}(X)$. Consequently each cone $\rho^\alpha$ can be represented by $\alpha$ and hence the $H$ functors in $\mathscr{P}(X)$ can be represented as $H(e;-)$. 

\begin{lem}\label{lem2}
Let $e \in  T_X $ and $A\subseteq X$. Then
$$ H (e ; {A}) = \{  a \in T_X \: : \: \pi_a \supseteq \pi_e \text{ and Im }a\subseteq A  \}$$
\end{lem}
\begin{proof}
By definition of $H$ functor ( equation \ref{eqnH}), we have
$$ H (e ; {A}) = \{ e \ast f^\circ \: :\: f : \text{Im } e \to A \} $$
Since $f: \text{ Im }e \to A$ is a morphism in the category $\bf{Set}$, $g\: := \: f^\circ$ will be a surjective morphism from Im $e \to$ Im $f \subseteq A$. Therefore  $ H (e ; {A}) = \{ e \ast g  \: : \: g:\text{ Im }e \to \text{ Im }f  \text{ is a surjection}\} = \{ u \: : \: u =eg  \} $. Since $u =eg$, Im $u = \text{ Im } eg \subseteq \text{ Im }g = \text{ Im }f \subseteq A$ .ie $\text{ Im }u \subseteq A$. Also as $u=eg$, by Lemma \ref{tx}, $\pi_u \supseteq \pi_e$. Thus $e \ast f^\circ\: = \: u \: \text{ where } \: \pi_u \supseteq \pi_e \text{ and Im }u\subseteq A$. \\
Conversely $a$ be such that $\pi_a \supseteq \pi_e \text{ and Im }a\subseteq A $. Then $a \: = \:{ea}$ (since $ea = a$ by lemma \ref{tx}). Take $h: =j(\text{ Im } a, A)$ and let $f=ah$; then $f^\circ = a$. Hence $a = e \ast f^\circ \text{ where } f: \text{Im } e \to A$.\\
Hence the lemma.
\end{proof}
\begin{lem}\label{lem3}
If $g : {A}\to{B} $ then $ H (e ; g ) : H(e;{A})  \mapsto  H(e;{B})  $.
\end{lem}
\begin{proof}
If $g : {A}\to{B} $ and $ H (e ; {A}) = \{  a \: : \: \pi_a \supseteq \pi_e \text{ and Im }a\subseteq A  \}$. Now $ (a) H (e ; g ) =  a \ast g^\circ $(by equation \ref{eqnH}). Let $h\: := \: g^\circ :$ $A \to$ Im $g \subseteq B$. Let $ah =b $ and as $ h:A \to$ Im $g$ , Im $b =$ Im $ah =$ Im $h = $ Im $g \subseteq B$ .i.e Im $b \subseteq B$. And since $ah=b$, by Lemma \ref{tx}, $\pi_a \subseteq \pi_b$; so $\pi_b \supseteq \pi_e$. So $(a) H (e ; g ) \subseteq \{  b \: : \: \pi_b \supseteq \pi_e \text{ and Im }b\subseteq B  \}$.\\
And by the previous lemma, $\{  b \: : \: \pi_b \supseteq \pi_e \text{ and Im }b\subseteq B  \} = H(e;B)$. Hence the lemma.
\end{proof}
Thus by Lemma \ref{lem2} and Lemma \ref{lem3}, we know that the $H$ functor $ H (e ;-)$ in $T_X$ is completely determined by the partition of $e$.\\
Now we proceed to show that there is a normal category isomorphism between $N^\ast \mathscr{P}(X)$ and $\Pi(X)$.
For that first we need to know the morphisms in $N^\ast \mathscr{P}(X)$. Note that these morphisms are natural transformations.
\begin{lem}\label{lem4} 
Let $\sigma: H(e;-) \to H(f;-)$ be a morphism in $N^\ast \mathscr{P}(X)$. Then the component of the natural transformation $\sigma$ is the map $$ \sigma({C}) : \{  a \: : \: \pi_a \supseteq \pi_e \text{ and Im }a\subseteq C  \} \to \{  b \: : \: \pi_b \supseteq \pi_{f} \text{ and Im }b\subseteq C  \}$$ given by $\sigma({C}) : a \mapsto va$ for $v \in fSe$.
\end{lem}
\begin{proof}
By Lemma \ref{lem2}, the $H(e; {C})\:=\:\{  a \: : \: \pi_a \supseteq \pi_e \text{ and Im }a\subseteq C  \}$. And given $\sigma$ between $ H(e;-) $ and $ H({f} ;-)$, by theorem \ref{thm1}, there exists a unique $\hat{\sigma}: \text{ Im }f \to \text{ Im }e$, say $v$ such that $e \ast h^\circ$ gets mapped to ${f} \ast (vh)^\circ$(See equation \ref{eqnH2}). As in the Lemmas \ref{lem2}, \ref{lem3}, we see that if $b = {f} \ast (vh)^\circ$, Im $b \subseteq \text{ Im } h^\circ \subseteq C$ . And since $\pi_{b} \supseteq \pi_{f}$, hence the lemma.\\
Also observe that as argued in the proofs of lemmas \ref{lem2}, \ref{lem3}; $\sigma(C):H(e;C) \to H({f} ;C)$ is a mapping $a \mapsto va$ for $v \in fSe$.
\end{proof}
Now we proceed to show that $N^\ast \mathscr{P}(X)$ is isomorphic to $\Pi(X)$ as normal categories.\\
Define a functor $P : N^\ast \mathscr{P}(X) \to \Pi(X) $ as
\begin{equation} \label{eqnP}
vP (H(e;-)) = \bar{\pi_e} \quad\text{and}  \quad
P(\sigma) = P_v
\end{equation}
where $\sigma: H(e;-) \to H(f;-)$ and $v = \hat{\sigma}: \text{ Im }f \to \text{ Im }e$ and $P_v : \bar{\pi_e} \to \bar{\pi_f} $. \\
Now we show that $P$ is a covariant functor.\\
\begin{lem} \label{lemP1}
$P$ as defined in equation \ref{eqnP} is a well defined covariant functor from $N^\ast \mathscr{P}(X)$ to $\Pi(X)$.
\end{lem}
\begin{proof}
If $H(e;-) = H({f} ;-)$, by Proposition \ref{proH} there exists a unique $h : \text{Im }f \to \text{Im }e $ such that $e = {f} \ast h $ .i.e $e = fu$ where $u = h^\circ$ and hence $\pi_{e} \supseteq \pi_{f} $ (by lemma \ref{tx}). Similarly there is $h^{-1} : \text{Im }e \to \text{Im }f $ ( $h$ being an isomorphism can be inverted ) such that ${f} = {e} \ast h^{-1} $. So $\pi_{f} \supseteq \pi_{e} $. Hence $\pi_{e} = \pi_{f} $. And so $\bar{\pi_e} = \bar{\pi_f}$. And $vP $ is well defined.\\
Now by lemma \ref{lem4}, $\sigma$ is determined by a unique $\hat{\sigma}: \text{Im }f \to \text{Im }e $, say $v$. This makes $P_v$ unique and the map $P$ is well defined on the morphisms as well. \\
Let $\sigma : H(e;-) \to H({f} ;-)$ and $\tau : H({f};-) \to H({g} ;-)$; then\\
$P(\sigma)P(\tau) = P_{v}P_u$(where $ v $ and $ u$ are chosen as discussed above) $ = P_{uv}$ (since $(\alpha)P_vP_{u} = ({v}\alpha)P_{u} = {u}\:{v}\alpha = {uv}\alpha = (\alpha)P_{uv} $).\\
Also $P(\sigma\circ\tau) = P_{uv}$ (By lemma \ref{lem4}, $(a)\sigma\circ\tau = ({va})\tau = {uva}$ ).\\
So $P(\sigma\circ\tau) = P(\sigma)P(\tau)$ and hence $P$ is a covariant functor. 
\end{proof}
Now define a functor $Q : \Pi(X) \to N^\ast \mathscr{P}(  X )$ as
\begin{equation} \label{eqnQ}
vQ (\bar{\pi}) = H(e;-) \quad\text{and}  \quad
Q(P_\eta) = \sigma
\end{equation}
where $e$ is the idempotent mapping such that $\pi_e = \pi$, ${\eta}$ is a function from $\pi_f$ to $\pi_e$ which uniquely determines a function $v$ from Im $f$ to Im $e$ and $\sigma: H(e;-) \to H({f} ;-) : a \mapsto {va} $. \\
\begin{lem} \label{lemQ1}
$Q$ as defined in equation \ref{eqnQ} is a well defined covariant functor from $\Pi(X)$ to $N^\ast \mathscr{P}( X )$.
\end{lem}
\begin{proof}
Suppose $\bar{\pi_1} = \bar{\pi_2}$ then $\pi_1 = \pi_2$. Suppose $e$ and $f$ are representative mappings of $\pi_1$ and $\pi_2$ respectively , then by lemma \ref{tx}, $e \mathscr{R} f$. By proposition \ref{proH} we have $H(e;-) = H(f;-)$. Hence $vQ$ is well defined.\\
Now suppose $P_\eta = P_\theta$. Let $v$ and $u$ be chosen as in the equation \ref{eqnQ} so that $P_\eta $ and $P_\theta$ determines $v$ and $u$ respectively; and let $\sigma$ and $\tau$ be the associated natural transformations. Since $P_\eta = P_\theta$, by lemma \ref{lemv1}, we have ${v} = {ue}$. And by lemma \ref{lem4}, $\sigma = \tau$. Hence $Q$ is well-defined on the morphisms as well. 
\end{proof}
\begin{lem}\label{lemP2}
$P$ is inclusion preserving.
\end{lem}
\begin{proof}
Suppose $\sigma$ be an inclusion in $N^\ast \mathscr{P}(  X )$ from $H(e;-)$ to $ H(f;-)$. By theorem \ref{thm1}, there is a unique epimorphism $e$ such that $e = f * e$. So $e = fe$ and $\pi_f \subseteq \pi_e$ and thus $\bar{\pi_e} \subseteq \bar{\pi_f}$. And $P_e = j(\bar{\pi_e},\bar{\pi_f})$ is an inclusion in $\Pi(X)$ and $P(\sigma) = P_e$ and hence $P$ is inclusion preserving. 
\end{proof}
\begin{lem} \label{lemP3}
$vP$ is an order isomorphism.
\end{lem}
\begin{proof}
Suppose $H(e ;-) \subseteq H({f} ;-)$, by theorem \ref{thm1}, this is true if and only if there exists a unique $\hat{\sigma}$ such that $e = {f} \ast \hat{\sigma}$ $ \iff e = {fu} $ $\iff \pi_{e} \supseteq \pi_{f} \iff \bar{\pi_e} \subseteq \bar{\pi_f}$ .\\ 
Hence $H(e ;-) \subseteq H({f} ;-) \iff \bar{\pi_e} \subseteq \bar{\pi_f}$ and $vP$ is an order isomorphism.
\end{proof}
\begin{thm} \label{thmP}
$N^\ast \mathscr{P}(  X )$ is isomorphic to $\Pi(X)$ as normal categories.
\end{thm}
\begin{proof}
Clearly $PQ = 1_{N^\ast \mathscr{P}( X)}$ and $QP = 1_{\Pi(X)}$. And using the previous lemmas \ref{lemP2}, \ref{lemP3}, $P$ is a category order isomorphism which preserves inclusions.
Hence the theorem.
\end{proof}

\end{document}